\documentclass[leqno]{amsart}
\usepackage{amsmath}
\usepackage{amssymb}
\usepackage{amsthm}
\usepackage{enumerate}
\usepackage[mathscr]{eucal}
\usepackage{graphicx}
\theoremstyle{plain}
\newtheorem{theorem}{Theorem}[section]

\newtheorem{cor}{Corollary}[theorem]
\newtheorem{lemma}{Lemma}[section]

\theoremstyle{definition}
\newtheorem{definition}{Definition}[section]
\newtheorem{remark}{Remark}[section]

\newtheorem{example}{Example}[theorem]

\usepackage[pagewise]{lineno}

\begin{document}

\title[On two extremum problems related to operator norm]{On two extremum problems related to the norm of a bounded linear operator}
\author[  Debmalya Sain,  Kallol Paul and Kalidas Mandal ]{  Debmalya Sain, Kallol Paul and Kalidas Mandal}

\newcommand{\acr}{\newline\indent}

\address[Sain]{Department of Mathematics\\ Indian Institute of Science\\ Bengaluru 560012\\ Karnataka \\India\\ }
\email{saindebmalya@gmail.com}

\address[Paul]{Department of Mathematics\\ Jadavpur University\\ Kolkata 700032\\ West Bengal\\ INDIA}
\email{kalloldada@gmail.com}

\address[Mandal]{Department of Mathematics\\ Jadavpur University\\ Kolkata 700032\\ West Bengal\\ INDIA}
\email{kalidas.mandal14@gmail.com}

\thanks{Research of Dr. Debmalya Sain is sponsored by Dr. D. S. Kothari Post-doctoral Fellowship under the mentorship of Professor Gadadhar Misra. Dr. Debmalya Sain feels elated to lovingly acknowledge the inspiring presence of his childhood friend Mr. Arijeet Mukherjee, in every sphere of his life! The research of Prof Kallol Paul  is supported by project MATRICS  of DST, Govt. of India. The third author would like to thank CSIR, Govt. of India for the financial support in the form of junior research fellowship. } 


\begin{abstract}
We explore the norm attainment set and the minimum norm attainment set of a bounded linear operator between Hilbert spaces and Banach spaces. Indeed, we obtain a complete characterization of both the sets, separately for operators between Hilbert spaces and Banach spaces. We also study the interconnection between these two sets and prove that for operators between Hilbert spaces, these two sets are either equal or mutually orthogonal, provided both of them are non-empty. We also obtain separate complete characterizations of reflexive Banach spaces and Euclidean spaces in terms of the norm (minimum norm) attainment set, in order to illustrate the importance of our study.
\end{abstract}

\maketitle
\section{Introduction.} 

The purpose of the present paper is to explore the norm attainment set and the minimum norm attainment set of a bounded linear operator between Hilbert (Banach) spaces. We would like to remark that such a study was initiated by Carvajal and Neves \cite{C,Ca}, for bounded linear operators between complex Hilbert spaces. However, our study has little intersection with theirs, and, moreover, we also explore the two problems for bounded linear operators between Banach spaces. Without further ado, let us discuss the notations and the terminologies relevant to our study.\\
 
 Let $\mathbb{X},~ \mathbb{Y}$ be normed spaces over the field $\mathbb{K}$, real or complex. We reserve the symbol $ \mathbb{H} $ for Hilbert spaces. Finite-dimensional real Hilbert spaces are called Euclidean spaces. Let $B_{\mathbb{X}} = \{x \in \mathbb{X} \colon \|x\| \leq 1\}$ and $S_{\mathbb{X}} = \{x \in \mathbb{X} \colon \|x\|=1\}$ be the unit ball and the unit sphere of $\mathbb{X}$, respectively. Let $\mathbb{X}^*$ denote the  dual space of $\mathbb{X}$. For any $f\in \mathbb{X}^*$, let $ker~f$ denote the null space of $f$. Given any $x\in S_{\mathbb{X}}$, a functional $\psi_x \in S_{\mathbb{X}^*}$ is said to be a support functional at $x$ if $\|\psi_x\|=1= \psi_x (x).$ It is easy to observe that the Hahn-Banach theorem guarantees the existence of at least one support functional at each point of $ S_{\mathbb{X}}. $ We say that $\mathbb{X}$ is smooth if given any $x\in S_{\mathbb{X}}$, there exists a unique support functional at $x$. $ \mathbb{X} $ is said to be strictly convex if for any $ x,y \in \mathbb{X},~ \| x+y \| = \| x \| + \| y \| $ implies that $ y = kx, $ for some $ k \geq 0. $ Next, we give the definition of semi-inner-products \cite{G,L} in normed spaces.
\begin{definition}
	Let $\mathbb{X}$ be a normed space. A function [ , ]: $\mathbb{X}\times \mathbb{X} \rightarrow \mathbb{K}(=\mathbb{C},\mathbb{R})$ is a semi-inner-product if and only if for any $\lambda \in \mathbb{K}$ and for any $x,y,z \in \mathbb{X}$, it satisfies the following properties:\\
	(i) $[x+y , z] = [x , z] +[y , z]$,\\
	(ii) $[\lambda x ,y] = \lambda [x , y]$,\\
	(iii) $[x , x]>0$, whenever $x\neq 0$,\\
	(iv) $|[x , y]|^2\leq [x , x][y , y]$,\\
	(v) $[x , \lambda y] = \overline \lambda [x , y]$.
\end{definition}

 Let $\mathbb{L}(\mathbb{X},\mathbb{Y})$ denote the normed space of all bounded linear operators from $\mathbb{X}$ to $\mathbb{Y}$, endowed with the usual operator norm. We write $\mathbb{L}(\mathbb{\mathbb{X}, \mathbb{Y}})= \mathbb{L}(\mathbb{X})$ if $\mathbb{X}= \mathbb{Y}$.

For any two elements  $x,y \in {\mathbb{X}}$, $x$ is said to be Birkhoff-James orthogonal to $y$ \cite{B}, written as $x \perp_B y$, if $ \|x+\lambda y\|\geq\|x\|$ for all
$ \lambda \in \mathbb{K} (=\mathbb{C},\mathbb{R}). $

For a bounded linear operator $T$ defined on a normed space $\mathbb{X}$, let $M_T$ denote the norm attainment set of $ T $, i.e., $ M_T $ is the collection of all unit vectors in $\mathbb{X}$ at which $T$ attains norm. To be more precise, 
\[ M_T= \{  x \in S_ \mathbb{X} \colon\|Tx\| = \|T\| \}. \]

Following similar motivations, we define the minimum norm attainment set $ m_T, $ for a bounded linear operator  $T$ defined on a normed space $ \mathbb{X}, $ in the following way: 
\[ m_T= \{  x \in S_ \mathbb{X} \colon\|Tx\| = m(T) \}, \] where, $ m(T)= inf \{ \|Tx\| \colon \|x\|=1 \}. $

For any two elements $ x, y $ in a real normed space $ \mathbb{X}, $ following \cite{Sd} we say that $ y \in x^{+} $ if $ \| x + \lambda y \| \geq \| x \| $ for all $ \lambda \geq 0. $ Similarly, we say that $ y \in x^{-} $ if $ \| x + \lambda y \| \geq \| x \| $ for all $ \lambda \leq 0. $ Let $x^{\perp}=\{y\in \mathbb{X}\colon x\perp_B y\}.$ These notions have been extended to complex normed linear spaces by Paul et.al. \cite{PSMM} in the following way: Let  $ x \in \mathbb{X}$ and $ U = \{ \alpha \in \mathbb{C} : | \alpha | = 1, ~\arg \alpha \in [0,\pi) \}.$ For $ \alpha \in U $ define 

\[ x_\alpha^{+}=\{y \in\mathbb{X}: \|x+\lambda y\|\geq\|x\|~ for ~all~ \lambda = t\alpha,  t\geq 0 \}, \]
\[ x_\alpha^{-}=\{y \in\mathbb{X}:\|x+\lambda y\|\geq\|x\|~ for~all~ \lambda = t\alpha,  t\leq 0 \},\] 
\[ x_\alpha^{\perp}  =  \{y \in\mathbb{X}:\|x+\lambda y\|\geq\|x\|~ for~all~ \lambda = t\alpha,   t\in\mathbb{R}\}.\]
If $ \beta = e^{i\pi} \alpha $ then we define $ x_\beta^{+} = x_\alpha^{-}, ~ x_\beta^{-} = x_\alpha^{+} $, $x_\beta^{\perp} = x_\alpha^{\perp}.$ If $ y \in x_\alpha^{\perp} $ then we write $ x \bot_{\alpha} y.$  The notions of $x^+$, $x^-$ and $x^{\perp}$ \cite{PSMM} are also defined   in a complex Banach space  in the following way:
 \[ x^{+}   =  \bigcap \{ x_\alpha^{+} : \alpha \in U \},  x^{-}  =  \bigcap \{ x_\alpha^{-} : \alpha\in U \} ,
 x^{\perp}  =  \bigcap \{ x_\alpha^{\bot} : \alpha\in U \} .\]                                                                                                         

\smallskip 

The norm attainment set plays a very crucial role in determining the geometry of the space of bounded linear operators \cite{S,Sa,SP}. Recently, Sain \cite{Sb} obtained a complete characterization of the norm attainment set of a bounded linear operator between real normed spaces, by applying the concept of semi-inner-products in normed spaces. In this paper, we extend the result for bounded linear operators on real or complex normed spaces. We also explore the minimum norm attainment set of a  bounded linear operator $ T $ between Hilbert spaces and Banach spaces. First, we obtain a complete characterization of $ m_T, $   for $ T \in \mathbb{L}(\mathbb{H}_1,\mathbb{H}_2), $ where $ \mathbb{H}_1,\mathbb{H}_2 $ are Hilbert spaces. We further explore the geometric structure of $ m_T $ for $ T \in \mathbb{L}(\mathbb{H}_1,\mathbb{H}_2), $ and obtain some interesting properties of $ m_T $ which are analogous to the properties of $ M_T. $ We observe that $ m_T $ must be the unit sphere of some subspace of $ \mathbb{H}_1, $ provided $ m_T $ is non-empty. We next obtain a complete characterization of the minimum norm attainment set of a bounded linear operator between real or complex normed spaces, analogous to the corresponding characterization of the operator norm attainment set. For $ T \in \mathbb{L}(\mathbb{H}_1,\mathbb{H}_2), $ we further study the relative position of $ M_T $ and $ m_T. $ In particular, we prove that if both $ M_T $ and $ m_T $ are non-empty, then either $ M_T = m_T = S_{\mathbb{H}_1} $ or $ M_T $ and $ m_T $ are the unit spheres of two subspaces of $ \mathbb{H}_1, $ which are mutually orthogonal. We would like to remark that in the first case, $ T $ is a scalar multiple of an isometry. On the other hand, as we will see later, the second condition is typical of bounded linear operators, which are not scalar multiples of some isometry, between Hilbert spaces.  We prove that for a rank one bounded linear operator $ T $ on a strictly convex reflexive Banach space $ \mathbb{X}, $ it is possible to describe $ M_T $ and $ m_T $ in a particularly convenient way. As an application of this observation, we obtain a complete characterization of reflexive Banach spaces in terms of the norm attainment sets and the minimum norm attainment sets of rank one bounded linear operators on the space. We end this paper with a characterization of Euclidean spaces among all finite-dimensional real Banach spaces, that further illustrates the importance of the study of the operator norm (minimum norm) attainment set. Let us further remark that for the two-dimensional case, we require the additional condition of strict convexity. \\

We would like to remark  that unless otherwise stated explicitly, we consider the Banach spaces and the Hilbert spaces to be  either real or complex.

\section{ Norm attainment set and Minimal norm attainment set}   
In this section, we first obtain a complete characterization of the norm attainment set for a bounded linear operator $T$ between normed linear spaces $ \mathbb{X}$ and $ \mathbb{Y}.$ We would like to remark that our result holds for both real and complex normed spaces and improves on \cite[Th. 2.3]{Sb}. In order to obtain the desired characterization, we need the following lemma, which again improves on \cite[Lemma 2.2]{Sb} in an elegant way.
\begin{lemma}\label{lemma:hyperplane1}
Let $\mathbb{X}, \mathbb{Y}$ be normed linear spaces and $T\in \mathbb{L}(\mathbb{X}, \mathbb{Y})$. Let $x\in M_T$  and $ y=Tx.$ Then there exist hyperspaces $H_x, H_y$ in $ \mathbb{X} $ and $ \mathbb{Y} $ respectively such that $x\perp_B H_x$ and $y\perp_B H_y$ with $T(H_x)\subseteq H_y$. 
\end{lemma}
\begin{proof}
	If $T$ is the zero operator then we have nothing to prove. Suppose $T$ is nonzero. Since Birkhoff-James orthogonality is homogeneous, and $T$ is nonzero, without any loss of generality we may assume that $\|T\|=1$.  	Let $x\in M_T .$  For $ y (= Tx) $, there exists a linear functional $g\in S_{\mathbb{Y}^*}$ such that $g(Tx)= \|Tx\|=1$. Let $ker~ g= H_y$. Then by Theorem 2.1 of \cite{Ja}, we have, $Tx\perp_B H_y$. 	Now, $g\circ T\colon \mathbb{X} \rightarrow \mathbb{K}$ is a linear functional with $g\circ T(x)=\|Tx\|=1=\|x\|$ and $\|g\circ T\|\leq \|g\|\|T\|=\|T\|=1$. Let $H_x= ker~(g\circ T)$. Again, by Theorem 2.1 of \cite{Ja}, we have, $x\perp_B H_x$.
	Let $h\in H_x$. Then $g\circ T(h)=0\Rightarrow g(Th)=0\Rightarrow Th\in H_y$. Since this is true for all $h\in H_x,$ we have, $T(H_x)\subseteq H_y$. This completes the proof of the lemma. 
\end{proof}

\begin{remark}\label{remark:nonsmooth}
	 The above result may not hold for all hyperspaces, i.e., given $ T\in \mathbb{L}(\mathbb{X}, \mathbb{Y}) $ and $ x \in M_T, $ there may exist a hyperspace $H$ in $ \mathbb{X} $ such that $x \perp_B H$ but $Tx\not\perp_B T(H)$. Let us illustrate the scenario by furnishing the following example.\\
 Consider $T\colon l_\infty(\mathbb{R}^2)\rightarrow l_\infty(\mathbb{R}^2)$ defined by
	\[ T(1,1)=(0,1),~~T(-1,1)=(-1,0).\]
Then $ M_T=\{\pm(1,1),\pm(-1,1)\}$.  Here, we have, $(1,1)\in M_T$ and $(1,1)\perp_B(0,1)$ but $T(1,1)=(0,1)\not\perp_B T(0,1)=(-\frac{1}{2},\frac{1}{2})$. Therefore, taking $ x = (1,1) $ and $ H $ to be the one-dimensional subspace spanned by $ (0,1) $, we see that $x\perp_B H$ but  $Tx\not\perp_B T(H)$.\\
Now, if we replace $\mathbb{R}^2$ by $\mathbb{C}^2$ in this example, then $ M_T=\bigcup_{\theta \in [0, 2\pi)} e^{i\theta}\{\pm(1,1),\pm(-1,1)\}$. Once again, we have,  $(1,1)\in M_T$ and $(1,1)\perp_B(0,1)$ but $T(1,1)=(0,1)\not\perp_B T(0,1)=(-\frac{1}{2},\frac{1}{2})$. Therefore, with the same choice of $ x $ and $ H, $ it follows that $x\perp_B H$ but  $Tx\not\perp_B T(H)$
\end{remark}

Let us now apply Lemma \ref{lemma:hyperplane1} towards obtaining a complete characterization of the norm attainment set of a bounded linear operator between normed spaces.

\begin{theorem}\label{theorem:characterization M_T}
Let $\mathbb{X}, \mathbb{Y}$ be normed linear spaces and $T\in \mathbb{L}(\mathbb{X}, \mathbb{Y})$. Let $x\in S_{\mathbb{X}}$. Then $x\in M_T$ if and only if there exist two s.i.p. $[~,~]_\mathbb{X}$ and $[~,~]_\mathbb{Y}$ on $\mathbb{X}$ and $\mathbb{Y}$ respectively such that for any $z\in \mathbb{X}$,\\ 
\[[Tz, Tx]_\mathbb{Y}= \|T\|^2[ z, x]_\mathbb{X}.\]
\end{theorem}
\begin{proof} If $T$ is the zero operator, then the theorem holds trivially. Without loss of generality we may assume that $\|T\| = 1.$ Let us first prove the sufficient part of the theorem. Let $x\in S_{\mathbb{X}}$ be such that for any $z\in \mathbb{X}$, 
	$[Tz, Tx]_\mathbb{Y}= \|T\|^2[ z, x]_\mathbb{X}.$ Taking $z=x$, we obtain, $[Tx, Tx]_\mathbb{Y}= \|T\|^2$ , i.e., $\|Tx\|=\|T\|$. This proves that $x\in M_T.$ \\
	Next we prove the necessary part.  Let $x\in M_T$ and $ y=Tx.$   Then from Lemma \ref{lemma:hyperplane1}, it follows that there exist hyperspaces $H_x, H_y$ in $ \mathbb{X} $ and $ \mathbb{Y} $ respectively such that $x\perp_B H_x$ and $Tx\perp_B H_y$ with $T(H_x)\subseteq H_y$. Since $x\perp_B H_x$, there exists a linear functional 
    $\psi_x\colon \mathbb{X}\rightarrow \mathbb{K}$ such that $\psi_x(x)=\|x\|$ and $ker ~\psi_x=H_x$. Since $T$ is nonzero and $x\in M_T$, we must have $Tx$ is nonzero in $\mathbb{Y}$. Therefore, there exists a linear functional 
    $\psi_{y}$ on $ \mathbb{Y} $ such that $\psi_{y}(y)=\|y\| = \|Tx\| = 1 $ and $ker~ \psi_{y}=H_y$. 
	 It follows from the Hahn-Banach theorem that for each $u \in S_{\mathbb{X}}$ and $ v \in S_{\mathbb{Y}}$ there exist at least one $f_u \in S_{\mathbb{X}^*}$ such that $ f_u(u) = 1$  and at least one $ g_v \in S_{\mathbb{Y}^*} $ such that $g_v(v)=1.$ 
Let us now define two s.i.p. on $\mathbb{X}$ and $\mathbb{Y}$ in the following way: 

\smallskip
For each $ z, u \in \mathbb{X},$ we define $ [z,u]_{\mathbb{X}} = f_u(z)$, with the additional restriction that if $ u = x$ then we take $ f_u= \psi_x.$ Moreover, for any $ \lambda \in \mathbb{K}, $ we choose $f_{\lambda u} 	= \bar{\lambda} f_u.$
 \smallskip
 		
For each $ w, v \in \mathbb{Y},$ we define $ [w,v]_{\mathbb{Y}} = g_v(w),$ with the additional restriction that if $ v = y$ then we take $ g_v= \psi_y.$ Moreover, for any $ \lambda \in \mathbb{K}, $ we choose $g_{\lambda v} 	= \bar{\lambda} g_v.$ 
\smallskip
					
Then following \cite{G}, it is easy to check that $[~,~]_{\mathbb{X}}$ and $ [~,~]_{\mathbb{Y}}$  are indeed  s.i.p. on $\mathbb{X}$ and $\mathbb{Y}$ respectively. 
Let $z\in \mathbb{X}$ be  arbitrary. Clearly $z$ can be written as $z=\alpha x + h$, for some $\alpha\in \mathbb{K}$ and $h\in H_x$. Moreover, we have, 
     \[[z,x]_{\mathbb{X}}=[\alpha x + h, x]_{\mathbb{X}}=\alpha[x,x]_{\mathbb{X}}+[h,x]_{\mathbb{X}}= \alpha \|x\|^2=\alpha.\]
Also,  
	\[[Tz,Tx]_{\mathbb{Y}}=[\alpha Tx + Th, x]_{\mathbb{Y}}=\alpha[Tx,Tx]_{\mathbb{Y}}+[Th,Tx]_{\mathbb{Y}}= \alpha \|Tx\|^2=\|T\|^2[z,x]_{\mathbb{X}}.\]
	Since the above relation holds for all $ z \in \mathbb{X},$ this completes the proof of the theorem.
\end{proof}
 
Let us now prove an easy but useful necessary condition for the minimum norm attainment of a nonzero bounded operator on a Banach space, at a particular point of the unit sphere. 

\begin{theorem}\label{theorem:preserve}
	Let $\mathbb{X}$ and $\mathbb{Y}$ be Banach spaces. Let $T \in \mathbb{L} (\mathbb{X}, \mathbb{Y})$ be non-zero and $x \in m_T$. Then
	
	(i) $T(x^{+})\subseteq (Tx)^{+}.$
	
	(ii) $T(x^{-})\subseteq (Tx)^{-}.$
	
	(iii) $T(x^{\perp})\subseteq (Tx)^{\perp}.$
	
\end{theorem}
\begin{proof}
	$(i)$ Let us assume that both $\mathbb{X}$ and $ \mathbb{Y}$ are complex Banach spaces. Let $y\in x^{+}.$ Then $ y \in x_{\alpha}^+ $ for each $ \alpha \in U. $ Then $ \| x + \lambda y \| \geq \|x \| = 1 $ for all $ \lambda = t \alpha, t \geq 0.$ Since $ x \in m_T, $ it follows that for any $ t \geq 0, $ we must have,
	$ \| Tx \| \leq \| T(\frac{x + t \alpha y}{ \| x + t \alpha y \| }) \|  = \frac{ \|Tx + t \alpha Ty \|} { \| x + t \alpha y \|} \leq \| Tx + t \alpha Ty \|.$ This implies that  $ Ty \in (Tx)_{\alpha}^+$ and so
	$T(x_{\alpha}^{+})\subseteq (Tx)_{\alpha}^{+}.$ This holds for each $ \alpha \in U$ and so  $T(x^{+})\subseteq (Tx)^{+}.$ \\
		For real Banach spaces the result follows by noting that $ x^+ = x_{\alpha}^+ $ with $\alpha=1.$ \\ 
(ii) and (iii) can be proved similarly.
		
\end{proof}

\begin{remark}
	It is interesting to observe that in Theorem \ref{theorem:preserve} $(iii), $ if we assume that $x\in M_T$ instead of assuming $ x \in m_T, $ then to prove the same result, we additionaly require smoothness of $x$ and $Tx$, which  follows from Theorem 2.3 of \cite{S}. Moreover, the example given in Remark \ref{remark:nonsmooth} shows that smoothness is necessary in case $ x \in M_T. $
\end{remark}

\begin{cor}
	Let $\mathbb{X}$ be a finite-dimensional Banach space and $T \in \mathbb{L} (\mathbb{X}).$ Then there exists $x \in S_{\mathbb{X}}$ such that $T$ preserves Birkhoff-James orthogonality at $x$, i.e., $x\perp_B y \Rightarrow Tx\perp_B Ty.$	
\end{cor}
\begin{proof}
	Since $\mathbb{X}$ is a finite-dimensional Banach space, there exists a unit vector $x_0$ such that $\|Tx_0\|=m(T)$, where $ m(T)= inf \{ \|Tx\| \colon \|x\|=1 \}.$ Therefore,  $x_0 \in m_T.$ Now, by Theorem \ref{theorem:preserve}, for any $y\in \mathbb{X}$, $x_0\perp_B y\Rightarrow Tx\perp_B Ty$, i.e., $T$ preserves Birkhoff-James orthogonality at $x_0.$
	
\end{proof}

An easy application of Theorem \ref{theorem:preserve} yields the following result.

\begin{cor}\label{cor:hyperplane2}
Let $\mathbb{X}, \mathbb{Y}$ be normed linear spaces and $T\in \mathbb{L}(\mathbb{X}, \mathbb{Y})$. If $x\in m_T$  then for any hyperspace $ H_x $ in $ \mathbb{X}, $ with $x\perp_B H_x, $ there exists a hyperspace $ H_y $ in $ \mathbb{Y}  $ such that $y(=Tx)\perp_B H_y$ with $T(H_x)\subseteq H_y$.
\end{cor}

\begin{proof}
Follows from Theorem \ref{theorem:preserve}$(iii).$
\end{proof}

As another application of Theorem \ref{theorem:preserve}, it is possible to slightly improve Theorem 2.8 of \cite{S}. Let $|M_T|$ denote the cardinality of $M_T$. It was proved in  \cite[Th. 2.8]{S} that if $T\in \mathbb {L}(l_{p}^2)$, where $p\in \mathbb{N}\setminus \{1\}$, is not a scalar multiple of an isometry then $|M_T|\leq 2(8p-5)$. Combining Theorem \ref{theorem:preserve} of the present paper with this result, we obtain the following theorem. 
\begin{theorem}
	Let $\mathbb{X}=l_{p}^2(\mathbb{R}), p\in \mathbb{N}\setminus \{1\}$ and let $T\in \mathbb{L}(\mathbb{X})$ be such that $T$ is not a scalar multiple of an isometry. Then $|M_T|\leq 4(4p-3)$. 
	
\end{theorem}
\begin{proof}
It follows from the arguments in the proof of \cite[Th. 2.8]{S}  that $T$ can preserve Birkhoff-James orthogonality at not more than $2(8p-5)$ number of points. Since $T$ is not a scalar multiple of an isometry, $M_T\bigcap m_T=\phi$. Furthermore, we have $m_T\neq \phi$, as $\mathbb{X}$ is finite dimensional. Since $m_T$ must contains at least $2$ elements, we must have, $|M_T|\leq 2(8p-5)-2=4(4p-3)$. This completes the proof of the theorem.
	
\end{proof}

Let us now obtain a complete characterization of the minimum norm attainment set of a bounded linear operator between Hilbert spaces. We would like to remark that the analogous characterization of the norm attainment set of a bounded linear operator between Hilbert spaces has been obtained in \cite{Sa}.

\begin{theorem}\label{theorem:characterization}
	Let $\mathbb{H}_1, \mathbb{H}_2$ be Hilbert spaces and $ T \in \mathbb{L}(\mathbb{H}_1, \mathbb{H}_2).$ Given any $x\in S_{\mathbb{H}_1}$, the following are equivalent:\\
	(i) $x\in m_T$.\\
	(ii) (a) Given any $ y \in H_1,~\langle x,y\rangle = 0$ implies that $\langle Tx, Ty\rangle = 0, $\\
	(b) inf $\{\|Ty\|\colon \|y\|=1, \langle x,y\rangle = 0\}\geq\|Tx\|.$\\
	(iii) $\langle Tx,Ty\rangle= m^2(T)\langle x,y\rangle,$ for every $y\in \mathbb{H}_1$.
\end{theorem}
\begin{proof}
	First, we prove $(i)\Rightarrow(ii)$. Suppose $x\in m_T$. Let $ y \in H_1 $ and $ \langle x,y\rangle = 0.$ Then from Theorem \ref{theorem:preserve} it follows that $\langle Tx, Ty\rangle = 0$. Thus (a) holds. Again from  the definition of $m(T)$, it follows that inf $\{\|Ty\|\colon \|y\|=1, \langle x,y\rangle = 0\} \geq m(T)=\|Tx\|.$  Therefore, (b) holds. \smallskip	
	Next, we prove $(ii)\Rightarrow(iii)$. 
	Let $ y \in H_1.$ Then $ y = \alpha x + h $ for some scalar $\alpha$ and $ h \in H_1$ such that $\langle x,h \rangle = 0.$ Then $ \langle Tx, Th \rangle = 0$ and  $ \langle Tx, Ty \rangle = \langle Tx, \alpha Tx + Th \rangle = \bar{\alpha} \langle Tx, Tx \rangle = 		m^2(T) \langle  x,  \alpha x \rangle = m^2(T) \langle x, y \rangle.$ \smallskip
	Finally we prove $(iii)\Rightarrow(i)$. Let $x\in S_{\mathbb{H}_1}$ be such that $\langle Tx, Ty\rangle=m^2(T)\langle x,y\rangle$ for every $y\in \mathbb{H}_1.$ Taking $y=x$, we get, $\|Tx\|^2=m^2(T),$ which implies  that $x\in m_T.$  	This establishes the theorem. 	
	\end{proof}
\begin{remark}
	It follows from Theorem 2.2 of \cite{SP} that for a bounded linear operator $ T $ between Hilbert spaces, if $ M_T $ is non-empty then $M_T $ is always the unit sphere of some subspace of the domain space. Applying the parallelogram equality, it is easy to see that the same fact holds for $ m_T. $ In other words, $m_T $ is also the unit sphere of some subspace of the domain space, provided it is non-empty.
\end{remark}

In Theorem \ref{theorem:characterization}, we proved that for a bounded linear operator $T$ on a Hilbert space $\mathbb{H}_1$, $x\in m_T$ if and only if $\langle Tx,Ty\rangle= m^2(T)\langle x,y\rangle,$ for every $y\in \mathbb{H}_1$. Since $m_T $ is always the unit sphere of some subspace of $ \mathbb{H}_1 $, using this characterization of $m_T$, we have the following theorem.

\begin{theorem}\label{theorem:dimension}
	Let $\mathbb{H}_1, \mathbb{H}_2$ be Hilbert spaces and $ T \in \mathbb{L}(\mathbb{H}_1, \mathbb{H}_2).$ The dimension of the subspace, whose unit sphere is $m_T$, is equal to the geometric multiplicity of the least eigen value (which is equal to $m^2(T)$) of $T^*T.$ 
\end{theorem}
\begin{proof}
	The proof of the theorem can be easily completed by following the same line of arguments, as used in  \cite[Th. 2.2]{Sa} for $M_T$.
\end{proof}

We next obtain a complete characterization of the minimum norm attainment set of a bounded linear operator between any two normed linear spaces. Let us mention that the following result is analogous to Theorem \ref{theorem:characterization M_T}.

\begin{theorem}\label{theorem:characterization m_T}
	Let $\mathbb{X}, \mathbb{Y}$ be normed linear spaces and $T\in \mathbb{L}(\mathbb{X}, \mathbb{Y})$. Let $x\in S_{\mathbb{X}}$. Then $x\in m_T$ if and only if there exist two s.i.p. $[~,~]_\mathbb{X}$ and $[~,~]_\mathbb{Y}$ on $\mathbb{X}$ and $\mathbb{Y}$ respectively such that for any $y\in \mathbb{X}$,\\ 
	\[[Ty, Tx]_\mathbb{Y}= m^2(T)[ y, x]_\mathbb{X}.\]
	
\end{theorem}

\begin{proof}
Let us first prove the sufficient part. Let $x\in S_{\mathbb{X}}$ such that for any $y\in \mathbb{X}$, 
$[Ty, Tx]_\mathbb{Y}= m^2(T)[ y, x]_\mathbb{X}.$ Taking $y=x$, we obtain, $[Tx, Tx]_\mathbb{Y}=m^2(T)$ , i.e., $\|Tx\|=m(T)$. However, this is clearly equivalent to the fact that $x\in m_T.$ \\
Let us now prove the necessary part. If $T$ is the zero operator, then it is clear that the theorem holds true. Suppose that $T$ is nonzero. Let $x\in m_T$. Let $y\in \mathbb{X}$ be arbitrary. If $m(T)=0$ then $\|Tx\|=0\Rightarrow Tx=0,$ and therefore, the theorem holds true. Suppose $m(T)>0$. Then applying Corollary \ref{cor:hyperplane2}, we can complete the proof of the theorem by using similar arguments, as done in the proof of Theorem \ref{theorem:characterization M_T}.
\end{proof}

\section{Relation between $M_T$ and $ m_T$ }

In this section we focus on studying the relation between $M_T$ and $m_T,$ both for bounded linear operators between Hilbert spaces as well as Banach spaces. We begin the study with a bounded linear operator $ T $ between Hilbert spaces $ \mathbb{H}_1 $ and $ \mathbb{H}_2. $ We note that in this case, both $M_T$ and $m_T$ are unit spheres of some subspaces of $ \mathbb{H}_1, $ provided they are non-empty. Indeed, our next theorem implies that these two subspaces are either identical or orthogonal to each other. We note that if $ T \in \mathbb L(\mathbb{H}_1,\mathbb{H}_2) $ is a scalar multiple of an isometry, then $ M_T = m_T = S_{\mathbb{X}}. $

\begin{theorem}\label{theorem:subset}
	Let $\mathbb{H}_1, \mathbb{H}_2$  be Hilbert spaces and let $T\in \mathbb L(\mathbb{H}_1,\mathbb{H}_2)$ be such that $T$ is not a scalar multiple of an isometry. Then $ m_T\subseteq (M_T)^\perp, $ provided both $ M_T $ and $ m_T $ are non-empty.
	\end{theorem}	
	
\begin{proof}
 Let us observe that since $T$ is not a scalar multiple of an isometry, we must have, $\|T\|>m(T).$ 	Let $y\in m_T$ be arbitrary. Choose $ x \in M_T$, which is chosen arbitrarily but is kept fixed after choice.  Since every Hilbert space is smooth, there exists a unique hyperspace $H_x$  such that $x\perp_B H_x$.  It is easy to see that $ y $ can be written as $y=\alpha x + h$, where $h\in H_x$ and $\alpha $ is a scalar.  If $\alpha = 0 $ then clearly $ y = h \in (M_T)^\perp$. If possible, suppose that $ \alpha \neq 0.$ 	Now, $1=\|y\|^2 =\langle\alpha x + h, \alpha x + h\rangle= |\alpha|^2 + \|h\|^2,$ since $\langle x,h\rangle=0.$ Moreover, from Lemma \ref{lemma:hyperplane1}, it follows that  $\langle Tx, Th\rangle=0.$  Now, we have,\\
	\begin{eqnarray*}
		\|Ty\|^2&=&\langle\alpha Tx + Th, \alpha Tx + Th\rangle\\
		&=&|\alpha|^2\|Tx\|^2 + \|Th\|^2\\
		&=&|\alpha|^2\|T\|^2 + \|h\|^2\|T(\frac{h}{\|h\|})\|^2\\
		&>&|\alpha|^2 m^2(T) + \|h\|^2m^2(T)\\
		&=&(|\alpha|^2 + \|h\|^2) m^2(T)=m^2(T).
	\end{eqnarray*}
	
However, this clearly contradicts  that $y\in m_T$. Therefore, we must have $\alpha=0$. Thus, for each $ x \in M_T, $ we get $ \langle y, x \rangle = 0 $ and so $ y \in (M_T)^{\bot}.$ As $ y \in m_T$ was chosen arbitrarily, this completes the proof of the theorem. 
	
\end{proof}

In particular, for a linear operator $T$ on a finite-dimensional Hilbert space $ \mathbb{H}, $ we have that either $ M_T = m_T = S_{\mathbb{H}} $ or  $ M_T\perp_B m_T $ and $ m_T\perp_B M_T. $ However, this is not true in general for a bounded linear operator between Banach spaces. Let us furnish the following two examples to illustrate the scenario. 

\begin{example}
	Consider $ (\mathbb {R}^2,\|\|), $ whose unit sphere is given by the regular hexagon with vertices at $ \pm (1,0), \pm (\frac{1}{2},\frac{\sqrt{3}}{2}),\pm (-\frac{1}{2},\frac{\sqrt{3}}{2}). $  \\
It is quite straightforward to observe that Birkhoff-James orthogonality is symmetric for this Banach space, though it is not an inner product space.\\
	Consider the linear operator $ T = $
	\(
	\left( 
	\begin{array}{cc}
	1 & 0 \\ 
	0 & 0
	\end{array} 
	\right)
	\) on this space. \\
	It follows immediately that $ \|T\|=1 $ and $ m(T)=0 $. It also follows that $ M_T = \{ \pm (1,0) \} $ and $ m_T = \{\pm (0, \frac{\sqrt{3}}{2}) \}. $ In this case, we indeed have, $M_T\perp_B m_T$ and $m_T\perp_B M_T$.
\end{example}

\begin{example}
Consider the same Banach space, as given in the previous example.

    Let $ T = $
	\(
	\left( 
	\begin{array}{cc}
	\frac{3}{4} & -\frac{\sqrt{3}}{4} \\ 
	\frac{\sqrt{3}}{4} & \frac{3}{4}
	\end{array} 
	\right)
	\).\\
	It follows immediately that $ \|T\|=1 $ and $m(T)=\frac{3}{4}$. It is also easy to check that $ \pm (1,0), \pm (\frac{1}{2},\frac{\sqrt{3}}{2}),\pm (-\frac{1}{2},\frac{\sqrt{3}}{2})\in M_T$ and $\pm (\frac{3}{4},\frac{\sqrt {3}}{4}), \pm (0, \frac{\sqrt{3}}{2}), \pm (-\frac{3}{4},\frac{\sqrt {3}}{4})\in m_T$. Therefore, in this case, $M_T\not\perp_B m_T.$\\
	
\end{example}

In the next theorem, we study the  norm (minimum norm) attainment set of a rank one linear operator on a (strictly convex) reflexive Banach space. As we will observe, this will lead us to an interesting characterization of reflexivity, in terms of these two sets.

\begin{theorem}\label{theorem:rank1}
	 Let $\mathbb{X}$ be a  reflexive  real Banach space and $Y$ be a real Banach space. Let $T\in \mathbb{L}(\mathbb{X}, \mathbb{Y})$ be a rank one linear operator. Then $x \in M_T$ for some $ x \in S_{\mathbb{X}} $ and   $m_T=H_x\bigcap S_{\mathbb{X}}$, where $H_x$ is a hyperspace of $\mathbb{X}$ such that $x\perp_B H_x$. In addition, if $\mathbb{X}$ is strictly convex, then $M_T=\{\pm x\}.$
\end{theorem}

\begin{proof} Without loss of generality, we may and do assume that $ \| T \| = 1. $ Since $T$ is a rank one operator on a reflexive space so $T$ attains norm at some element $x \in S_{\mathbb{X}}. $ Let $ y = Tx.$ Then by Lemma \ref{lemma:hyperplane1}, there exist hyperspaces $H_x$ and $H_y$ in $\mathbb{X}$ and $\mathbb{Y}$ respectively such that $ x \bot_B H_x, Tx \bot_B H_y $ and $ T(H_x) \subset H_y.$ We further note that $ Tx \neq 0. $
We claim that $ Tz = 0$ for all $z \in H_x.$ If not, then as $Tx \bot_B Tz,$ we must have, $\{Tx,Tz\}$ is linearly independent in $\mathbb{Y}.$ However, this implies that the rank of $T$ is more than one, a contradiction to our hypothesis. Thus, $ Tz =0 $ for all $z \in H_x$ and so $ H_x \cap S_{\mathbb{X}} \subset m_T.$ Next, let $z \in m_T.$ Then $ z = \alpha x + h $ for some scalar $\alpha$ and $ h \in H_x$. Clearly, $ 0 = Tz = \alpha Tx + Th = \alpha Tx,$ so that $\alpha = 0 $ and hence $ z = h \in
H_x \cap S_{\mathbb{X}}.$ Thus, $ m_T \subset H_x \cap S_{\mathbb{X}}.$ This proves that $m_T=H_x\bigcap S_{\mathbb{X}}, $ and completes the proof of the fist part of the theorem.
\smallskip
 
	Next, assume that $ \mathbb{X}$ is strictly convex. We show that $M_T=\{\pm x\}$. Clearly, $w \in S_{\mathbb{X}}$ can be written as $w=\alpha x + h, $ for some scalar $\alpha $ and some $  h\in H_x.$ Since $\mathbb{X}$ is strictly convex and  $x\perp_B h$, we have, $1=\|w\|=\|\alpha x + h\| \geq |\alpha|$ and $|\alpha|=1$ if and only if $ h=0$. Now, $\|Tw\|=\|T(\alpha x + h)\|=|\alpha|\|Tx\|=|\alpha|\leq 1$ and equality holds if and only if $h=0$. Therefore, we must have $M_T=\{\pm x\}$. This completes the proof of the theorem.
	
\end{proof}

Now, the promised characterization of reflexive Banach spaces:

\begin{theorem}\label{theorem:reflexive}
	Let $\mathbb{X}$ be a real Banach space. Then $\mathbb{X}$ is reflexive if and only if for any closed hyperspace $H$ of $\mathbb{X}$, there exists a rank one linear operator $T\in \mathbb{L}(\mathbb{X}) $ such that \\
	(i) $ x \in M_T $, for some $x \in S_{\mathbb{X}}.$\\
	(ii) $m_T= H \bigcap S_{\mathbb{X}}.$ 
\end{theorem}
\begin{proof}
	We first prove the necessary part. Since $\mathbb{X}$ is reflexive, it follows from \cite{Jam} that for any closed hyperspace $H$ of $\mathbb{X},$  there exists a unit vector $x\in \mathbb{X}$ such that $x\perp_B H.$ Clearly, any element $z\in \mathbb{X}$ can be written as $z=\alpha x + h$, where $\alpha \in \mathbb{R},~ h\in H.$ Define $T : \mathbb{X} \longrightarrow \mathbb{Y} $ as $T (\alpha x + h) = \alpha y ,$ where $y \in S_{\mathbb{X}} $ is fixed.  Clearly, $T$ is well-defined and $ T $ is a rank one linear operator. Since $x\perp_B H$, it is easy to check that is $T$ bounded and $x \in M_T.$ So from Theorem \ref{theorem:rank1} it follows that $m_T= H\bigcap S_{\mathbb{X}}.$ This completes the proof of the necessary part.
We next prove the sufficient part. Let $H$ be a closed hyperspace of $\mathbb{X}.$ According to our hypothesis, there exists a rank one linear operator $T\in \mathbb{L}(\mathbb{X})$ such that \\
	$(i)~ x \in M_T, $ for some $x \in S_{\mathbb{X}}.$\\
	$(ii)~ m_T= H \bigcap S_{\mathbb{X}}.$\\
Since rank of $ T $ is one, it is immediate that $ m(T) = 0. $ Since $m_T= H \bigcap S_{\mathbb{X}}, $ it follows that $ Th = 0 $ for all $ h \in H. $ In particular, we have that $ x \in M_T $ and $ Tx \perp_{B} Th $ for all $ h \in H. $ Applying Proposition 2.1 of \cite{S}, it now follows that  $ x \perp_{B} h $ for all $ h \in H, $ i.e., $ x \perp_{B} H. $ Thus, for each closed hyperspace $H$ of $ \mathbb{X}, $ there exists an element $ x \in S_{\mathbb{X}}$ such that $ x \bot_B H. $ Therefore, it follows from \cite{Jam} that $ \mathbb{X} $ is reflexive. This completes the proof of the sufficient part and establishes the theorem in its entirety. 

\end{proof}

In addition, if we assume that the space $\mathbb{X}$ is strictly convex, then we have the following theorem, the proof of which follows trivially from the previous theorem and the last part of Theorem \ref{theorem:rank1}. 

\begin{theorem}\label{theorem:reflexive2}
	Let $\mathbb{X}$ be a strictly convex real Banach space and $ \mathbb{Y}$ be a real Banach space. Then $\mathbb{X}$ is reflexive if and only if for any closed hyperspace $H$ of $\mathbb{X}$, there exists a rank one linear operator $T\in \mathbb{L}(\mathbb{X}, \mathbb{Y})$ such that \\
	(i) $  M_T = \{ \pm  x \}$, for some $x \in S_{\mathbb{X}}.$\\
	(ii) $m_T= H \bigcap S_{\mathbb{X}}.$ 
\end{theorem}

Our next objective is to characterize Euclidean spaces among finite-dimensional real Banach spaces, in terms of the norm attainment set and the minimum norm attainment set of bounded linear operators on them. We first prove the following result for two-dimensional strictly convex Banach spaces.

\begin{theorem}\label{theorem:ips dim 2}
	A two-dimensional strictly convex real Banach space $\mathbb{X}$ is an inner product space if and only if for any $T \in \mathbb{L}(\mathbb{X})$, either (a) or (b) holds:\\
	(a) $M_T = m_T = S_{\mathbb{X}}$. \\
	(b) $M_T \perp_B m_T$ and $m_T \perp_B M_T.$	
\end{theorem}

\begin{proof}
	Let us first prove the necessary part. Let $ \mathbb{X} $ be the two-dimensional Euclidean space. Let $T \in \mathbb{L}(\mathbb{X}).$   If $T$ is a scalar multiple of an isometry then $M_T = m_T = S_X$. On the other hand, if $T$ is not a scalar multiple of an isometry then it follows from Theorem \ref{theorem:subset} that $M_T \perp_B m_T$ and $m_T \perp_B M_T.$ This completes the proof of the necessary part of the theorem. Let us now prove the sufficient part. We first claim that for any $T \in L(\mathbb{X}), $  $M_T= \pm D$, where $D$ is a connected subset of $S_{\mathbb{X}}.$ If (a) holds, i.e., $M_T = S_{\mathbb{X}},$ then our claim is trivially true. 
	Next, suppose $(b)$ holds. We show that $T$ attains norm at only one pair of points. If possible, suppose that $x$, $y \in M_T$, where $x\neq \pm y. $ Let $z\in m_T$. Clearly, $ z \neq \pm x,~ \pm y, $ as $ T $ is not a scalar multiple of an isometry. Therefore, $z$ can be written as $z=\alpha x +\beta y$, where $\alpha ,\beta $ are  non-zero scalars. We have, $x \perp_B z$ and $y \perp_B z$. Since $\mathbb{X} $ is a two-dimensional strictly convex Banach space, it follows from \cite{J} that Birkhoff-James orthogonality is left additive in $ \mathbb{X}. $ Therefore, applying the homogeneity property of Birkhoff-James orthogonality, it follows that $\alpha x +\beta y\perp_B z$, i.e., $z\perp_B z$, which is possible only if $ z=0. $ However, this clearly contradicts that $ z \in m_T \subset S_{\mathbb{X}}. $ This completes the proof of the fact that if $ (b) $ holds then $T$ must attain norm only at one pair of points. Therefore, in any case, $M_T= \pm D$, where $D$ is a connected subset of $S_{\mathbb{X}}.$ It now follows from \cite[Th. 2.2]{SP} that $\mathbb{X}$ is an inner product space. This completes the proof of the sufficient part of the theorem and thereby establishes the theoerem. \\ 
\end{proof}

\begin{remark}
Let $\mathbb{X}$ be a two-dimensional real Banach space which is not strictly convex. Then the unit sphere of $S_{\mathbb{X}}$ contains a line segment $L$(say). Let $x\in L.$ It is easy to see that there exists $y\in S_{\mathbb{X}}$ such that every point of $ L $ is Birkhoff-James orthogonal to $ y. $ Let us define a linear operator $T$ on $ \mathbb{X} $ in the following way: $ Tx=x, Ty=0.$ It follows trivially that $M_T = \pm L $ and $ m_T = \{ \pm y \}. $ It is also immediate that $M_T\perp_B m_T$ but $m_T$ may not be always  Birkhoff-James orthogonal to $M_T$. In particular, if we further assume $ \mathbb{X} $ to be smooth, then it follows that $ \mathbb{X} $ is not an inner product space. 
\smallskip

As for example, consider the linear operator $T$ defined on $ \ell_\infty({\mathbb{R}^2}) $ as $ T(1,0) = (1,0) $ and $ T(0,1) = (0,0).$ Then it easy to check that $ M_T = \{ (a,b) ~:~ \mid a \mid = 1, \mid b \mid \leq 1 \} $ and $m_T = \{\pm (0,1) \}.$ Clearly, $ M_T \bot_B m_T $ but $ m_T \not\perp_B M_T.$
\end{remark}

If the dimension of $ \mathbb{X} $ is strictly greater than $ 2, $ then we have the following characterization of Euclidean spaces.

\begin{theorem}\label{theorem:ips}
	Let $ \mathbb{X} $ be a finite-dimensional real Banach space having dimension strictly greater than $ 2. $ Then  $\mathbb{X}$ is an Euclidean space if and only if for any $T \in \mathbb{L}(\mathbb{X})$, either (a) or (b) holds:\\
	(a) $M_T = m_T = S_{\mathbb{X}}$. \\
	(b) $M_T \perp_B m_T$ and $m_T \perp_B M_T.$
\end{theorem}

\begin{proof}
	We note that the proof of the necessary part of the theorem follows similarly as that of the necessary part of Theorem \ref{theorem:ips dim 2}. Let us prove the sufficient part.  We claim  that Birkhoff-James orthogonality is symmetric in $ \mathbb{X}. $ Let $x, y \in S_{\mathbb{X}}$ be such that $x\perp_B y$. Then there exists a hyperplane $H$  containing $y$ such that $x\perp_B H$. Clearly, any element $z\in \mathbb{X}$ can be written as $z=\alpha x + h$, where $\alpha \in \mathbb{R},~ h\in H.$ Define a linear operator $T$ on $\mathbb{X}$ as follows:
	\[T(\alpha x + h) = \alpha x,~ \text{for each}~ \alpha \in \mathbb{R}~ \text{and for each}~ h\in H.\]
	Clearly, $T$ is well-defined, linear and bounded. Since $x\perp_B H$, it is easy to check that $x \in M_T$ and $ y\in m_T.$ Clearly, $M_T\neq S_{\mathbb{X}}$. Therefore, $(b)$ holds. So, we have, $y\perp_B x$. Since $x, y \in S_{\mathbb{X}}$ such that $x\perp_B y$ was chosen arbitrarily, it follows from the homogeneity property of Birkhoff- James orthogonality that  Birkhoff- James orthogonality is symmetric in $\mathbb{X}.$ Since the dimension of $ \mathbb{X} $ is strictly greater than $ 2, $ it follows from \cite{J} that $\mathbb{X}$ is an inner product space. This establishes the theorem.
	
\end{proof}

\end{document}